\documentclass[12pt,reqno]{amsart}
\usepackage{amstext}
\usepackage{amsthm}
\usepackage{amsmath}
\usepackage{amssymb}
\usepackage{latexsym}
\usepackage{amsfonts}
\usepackage{graphicx}
\usepackage{texdraw}
\usepackage{graphpap}

\usepackage[pagebackref,hypertexnames=false, colorlinks, citecolor=red, linkcolor=red]{hyperref} 
\usepackage[backrefs]{amsrefs}

\usepackage{thmtools}
\usepackage{thm-restate}
\usepackage{cleveref}

\input txdtools

\begin{document}

\newcommand{\norm}[1]{\ensuremath{\left\|#1\right\|}}
\newcommand{\abs}[1]{\ensuremath{\left\vert#1\right\vert}}
\newcommand{\ip}[2]{\ensuremath{\left\langle#1,#2\right\rangle}}
\newcommand{\p}{\ensuremath{\partial}}
\newcommand{\pr}{\mathcal{P}}

\newcommand{\pbar}{\ensuremath{\bar{\partial}}}
\newcommand{\db}{\overline\partial}
\newcommand{\D}{\mathbb{D}}
\newcommand{\B}{\mathbb{B}}
\newcommand{\Sp}{\mathbb{S}}
\newcommand{\T}{\mathbb{T}}
\newcommand{\R}{\mathbb{R}}
\newcommand{\Z}{\mathbb{Z}}
\newcommand{\C}{\mathbb{C}}
\newcommand{\N}{\mathbb{N}}

\newcommand{\supp}{\operatorname{supp}}
\newcommand{\spa}{\operatorname{span}}

\def\to{\rightarrow}
\def\La{{\Leftarrow}}
\def\Ra{{\Rightarrow}}

\def\no{\noindent}

\def\PWa{\mathcal{P}\mathcal{W}_a}

\def\AA{{\mathcal{A}}}
\def\BB{{\mathcal B}}
\def\CC{{\mathcal C}}
\def\DD{{\mathcal{D}}}
\def\EE{{\mathcal E}}
\def\FF{{\mathcal F}}
\def\GG{{\mathcal{G}}}
\def\HH{{\mathcal H}}
\def\II{{\mathcal{I}}}
\def\JJ{{\mathcal{J}}}
\def\LL{{\mathcal{L}}}
\def\MM{{\mathcal M}}
\def\NN{{\mathcal N}}
\def\PP{{\mathcal{P}}}
\def\QQ{{\mathcal{Q}}}
\def\RR{\mathcal{R}}
\def\SS{{\mathcal{S}}}
\def\TT{{\mathcal{T}}}
\def\UU{{\mathcal{U}}}
\def\XX{{\mathcal{X}}}
\def\YY{{\mathcal{Y}}}
\def\ZZ{{\mathcal{Z}}}

\def\WW{{\mathcal W}}
\def\KK{{\mathcal K}}
\def\BM{{\mathcal{BM}}}

\def\a{{\alpha}}
\def\e{{\epsilon}}
\def\d{{\delta}}
\def\D{{\Delta}}
\def\l{{\lambda}}

\def\o{{\omega}}
\def\sL{{\sigma_{\Lambda}}}
\def\sG{{\sigma_{\Gamma}}}

\def\lan{{\lambda_{n}}}
\def\lam{{\lambda_{m}}}
\def\lak{{\lambda_{k}}}
\def\lmn{{\lambda_{mn}}}
\def\gmn{{\gamma_{mn}}}

\def\k0{{\xi_{0}}}
\def\ms{{\mu^*}}

\def\Th{{\Theta}}
\def\L{{\Lambda}}
\def\G{{\Gamma}}

\def\lfl{{\left \lfloor}}
\def\rfl{{\right \rfloor}}
\def\mh{{\hat{\mu}}}
\def\nh{{\hat{\nu}}}
\def\sh{{\hat{\sigma}}}
\def\mup{{\sigma^+}}
\def\mum{{\sigma^-}}
\def\nup{{\nu^+}}
\def\num{{\nu^-}}

\def\const{\text{\rm const}}
\def\ti{\tilde}

\title{Determinacy for measures}
\subjclass[2010]{}
\keywords{}

\begin{abstract} We study the general moment problem for measures on the real line, with polynomials replaced by more general spaces
of entire functions. As a particular case, we describe measures that are uniquely determined by a restriction of their Fourier transform
to a finite interval. We apply our results to prove an extension of a theorem by Eremenko and Novikov on the frequency of oscillations of
measures with a spectral gap (high-pass signals) near infinity.
\end{abstract}


\author[M. Mitkovski]{Mishko Mitkovski$^\dagger$}
\address{Mishko Mitkovski, Department of Mathematical Sciences\\ Clemson University\\ Clemson, SC USA }
\email{mmitkov@clemson.edu}
\urladdr{http://people.clemson.edu/$\sim$mmitkov}
\thanks{$\dagger$ Research supported in part by National Science Foundation DMS grant \# 1101251.}

\author[A. Poltoratski]{Alexei Poltoratski$^\ddagger$}
\address{Alexei Poltoratski, Department of Mathematics\\ Texas A\&M University\\ College Station, TX USA}
\email{alexeip@math.tamu.edu}
\urladdr{www.math.tamu.edu/$\sim$alexeip}
\thanks{$\ddagger$ Research supported in part by National Science Foundation DMS grant \# 1101278. }

\subjclass[2000]{}
\keywords{Determinacy, spectral gap, Fourier transform, moment problem}

\maketitle

\theoremstyle{plain}
\newtheorem{theorem}{Theorem}[section]
\newtheorem{lemma}[theorem]{Lemma}
\newtheorem{corollary}[theorem]{Corollary}
\newtheorem{proposition}[theorem]{Proposition}
\newtheorem{example}[theorem]{Example}
\newtheorem{problem}[theorem]{Problem}
\newtheorem{remark}[theorem]{Remark}

\theoremstyle{definition}
\newtheorem{definition}[theorem]{Definition}
\newtheorem*{remark*}{Remark}

\newenvironment{proofA}
{{\noindent{\bf \textit{Proof of Theorem 1.2:}}}}{\hfill$\Box$}

\numberwithin{equation}{section}
\section{Introduction}

For a finite measure $\mu$ on the real line we define its Fourier transform as
$$\mh(x)=\int e^{ixt}d\mu(t).$$
We say that a positive finite measure $\mu$ is $a$-determinate if there exists no other positive finite measure $\nu$ such that  Fourier transforms of $\mu$ and $\nu$ coincide on $[-a,a]$.
One of the  problems that we consider here can be formulated as follows.

Suppose that we are given a finite positive measure $\mu$. For a given $a>0$, how can we tell whether $\mu$ is $a$-determinate?
This question can be viewed as the analog of the determinacy part of the classical moment problem, which received a lot of attention throughout the years.  The study of the trigonometric version of the  problem stated above originates from  a classical paper of Krein \cite{Krein3}.
Despite a number of deep results, it appears that no explicit solution is known, even in the classical polynomial case.

In this paper we consider the general case of determinacy in the moment problem, with polynomials (exponentials) replaced with broader classes of entire functions.
The correct space for such purposes seems to be the de Branges space, which can become a space of polynomials or a Paley-Wiener space with a proper choice of a generating Hermite-Biehler function $E$.

We apply our methods to generalize classical theorems by Beurling, de Branges, and Levinson on the absence of gaps in the support of the Fourier transform (spectral gaps) of a finite measure on $\R$. Another application is a refinement of a theorem by Eremenko and Novikov \cite{EN, EN1}
on the oscillations of high-pass signals.

Let $\mu$ be a finite real measure on $\R$ with the property that its Fourier transform $\hat\mu$ vanishes on the interval $[-a,a]$ for some $a>0$.
Such measures are called high-pass signals in electrical engineering. One of the classical problems  is to estimate
how fast a high-pass signal should oscillate near infinity. Significant progress in this direction was provided by the following result.

\begin{restatable}{theorem}{ErNov}\cite{EN, EN1}
 If $\sigma$ is a nonzero signed measure with spectral gap $(-a/2, a/2)$  then the number of sign changes $s(r,\sigma)$ of $\sigma$ on the interval $(0,r)$ satisfies
$$\liminf_{r\to\infty}\frac{s(r,\sigma)}{r}\geq \frac{a}{2\pi}.$$
\end{restatable}

Here the number of sign changes $s(r,\sigma)$ can be understood in the usual sense for absolutely continuous measures with continuous densities and in any reasonable sense for more general
measures, see section \ref{last}. This theorem solves an old problem by Grinevich from 1964 included in V. Arnold's list of problems (2000 \cite{Arn}). A problem of achieving a better understanding of the asymptotic interplay between the positive and negative parts of a high-pass signal, including a possibility of replacing the inequality in the last theorem with an equality of some sort, served as an initial motivation for the present paper.

In this article we obtain a short and self-contained proof of the last statement, that does not rely on any of the  deep tools of Harmonic Analysis
often applied in this area, such as the Beurling-Malliavin theory or advanced estimates of singular integrals. At the same time, with the use of such advanced tools, we are able to prove the following stronger statement.

 For a finite real measure $\mu$ on $\R$ we denote by $\mu^+$ and $\mu^-$ its mutually singular positive and negative parts, $\mu=\mu^+-\mu^-,\ \mu^+\perp\mu^-$.
 If $A$ and $B$ are two closed subsets of $\R$,  let
$\MM_a(A,B)$ be the class of all finite real measures $\sigma$ with a spectral gap $[-a,a]$ such that $\supp \sigma^+\subset A$ and $\supp \sigma^-\subset B$. We define the gap characteristic of a pair of closed subsets $A$ and $B$ of $\R$ as
$$G(A,B)=\sup\{a>0: \MM_a(A,B)\neq\{0\}\}.$$

\begin{restatable}{theorem}{oscillation} \label{oscillation} For any closed sets $A, B\subset\R$,
$$ G(A,B)=\pi\sup\{d: \exists\ \   d\text{-uniform sequence }\{\lambda_n\},\ \{\lambda_{2n}\}\subset A, \{\lambda_{2n+1}\}\subset B\}.$$
\end{restatable}

The definition of a $d$-uniform sequence was given in \cite{PolGap} and is discussed in the next section.
One can say that, in a certain sense made precise in section \ref{prelim}, $d$-uniform sequences are similar to arithmetic progressions with a difference term equal to $1/d$.
Such sequences are discrete (i.e. have no accumulation points on $\R$) and are enumerated in the natural increasing order.

Theorem~\ref{oscillation} follows from a general discussion of $a$-determinacy of measures and properties of extreme indeterminate measures
that we undertake in sections \ref{Riesz} and \ref{extreme}. First, we prove a generalization of a classical result by M. Riesz for the
polynomial moment problem that characterizes indeterminate measures in terms of the norms of point-evaluation functionals.
A similar theorem was recently presented by M. Sodin~\cite{Sod}.
We then apply that result to extend classical gap theorems.

Following Beurling and Malliavin, we say that a union of disjoint open intervals $\bigcup (a_n,b_n)$ is long if $$\sum_n \left(\frac{b_n-a_n}{a_n}\right)^2=\infty.$$
The classical Beurling's gap theorem~\cite{Beu} states that if the complement of $\supp\mu$ is long than $\mu$ does not have
any spectral gaps, unless $\mu\equiv 0$. Like most results by Beurling, this statement is sharp in its scale:
the condition of longness of the complement of the support cannot be weakened, as was proved by M. Benedicks~\cite{Ben}.
However, as we show in section \ref{Riesz}, the result can be strengthened in the following way:

\begin{restatable}{corollary}{gap}\label{gap} If $\mu$ is a positive finite measure which is supported on a set whose complement is long then $\mu$ is $a$-determinate for every $a>0$.
\end{restatable}
\no   Equivalently, if $\mu$ is a real finite non-zero measure and the complement of $\mu_+$ is long, then $\mu$ has no spectral gaps. Notice that we do not put any conditions on the support of the negative part of the measure.  Other classical gap theorems by de Branges and Levison can be similarly
strengthened, see section \ref{Riesz}.

The last statement implies that if $\mu>0$ is $a$-indeterminate for some $a>0$ then $\mu$ must contain a sequence of positive interior Beurling-Malliavin density in its support. It is natural to expect further relations of this type between the  determinacy of  $\mu$ and density of its support. Our next result makes this observation precise.

\begin{restatable}{theorem}{indeterminate} \label{indeterminate} Let $\mu$ be a positive finite measure. If $\mu$ is $a$-indeterminate then the support of $\mu$ contains an $a/2\pi$-uniform sequence.
\end{restatable}

Extending the last statement, we prove that for a general regular de~Branges space $\BB_E$ a positive measure $\mu$ whose support is sufficiently sparse (in a certain precise sense) is uniquely determined by its action on the space, meaning that there is no other positive measure $\nu$ such that
$$\int F(t)d\mu(t)=\int F(t) d\nu(t), $$ for all $F\in\BB_E$.

Clearly, every $a$-indeterminate measure is a positive part of a signed measure $\sigma$ with a spectral gap $(-a,a)$, i. e.,  such that $\hat{\sigma}(x)=0,$ for $x\in [-a,a]$. Conversely, every non-zero measure $\sigma$ with a spectral gap $(-a,a)$ gives a rise to two $a$-indeterminate measures $\sigma^+$ and $\sigma^-$. Therefore, the determinacy problem is closely related to the gap problem which was recently studied by the second author in~\cite{PolGap}. We prove the following result that gives a precise quantitive relation between these two problems.

For a closed set $X\subset \R$ its gap characteristic is defined by $$G(X)=\sup\{a: \exists\text{ a real non-zero measure } \mu  \text{ with spectral gap } [-a, a]\text{ supported on } X\}.$$ Define the determinacy characteristic of $X$ by
$$Det(X)=\inf\{a : \text{all } \mu \text{ supported on } X \text{ are } a\text{-determinate}\}.$$

\begin{restatable}{theorem}{DetGap}\label{DetGap} For any closed set $X\subset\R$, $$2Det(X)=G(X).$$\end{restatable}

\noindent\textbf{Acknowledgement.} We are grateful to A. Eremenko and M. Sodin who brought the problem on oscillations of Fourier integrals to our attention.

\section{Preliminaries}\label{prelim}

We  denote by $\Pi$ the Poisson measure on $\R$, $d\Pi(x)=dx/(1+x^2)$.
The notation $\CC_a$ stands for the Cartwright class, the class of entire functions $F$ of exponential type no greater than $a$  such that $\log|F|\in L^1(\Pi)$. We will also use the standard notation $\PWa$ for the Paley-Wiener class, and $\BB_a$ for the Bernstein class consisting of entire functions of type at most $a$ which are bounded on $\R$. Finally, we will denote by $\KK_a$ the Krein class  consisting of entire functions $F(z)$ of exponential type no greater than $a$ which have only real simple zeros $\{\lambda_n\}$, such that $\sum_n1/|F'(\lambda_n)|<\infty,$ and satisfy $$\frac{1}{F(z)}=\sum_n\frac{1}{(z-\lambda_n)F'(\lambda_n)}.$$  By a well known theorem of M. Krein each function from $\KK_a$ belongs in $\CC_a$

If $E$ is an entire function we denote $E^\#(z)=\bar E(\bar z)$. An entire function $E$ belongs to the Hermite-Biehler class if
it satisfies
$$|E(z)|>|E^\#(z)|$$
for all $z\in \C_+=\{\Im z>0\}$. Hermite-Biehler functions are often called de Branges functions in the literature.
For any Hermite-Biehler (de Branges) function $E$, the de Branges space $\BB_E$ is defined as a space of all
entire functions $F$ such that both $F/E$ and $F^\#/E$ belong to the Hardy space $H^2(\C_+)$.
For more on de Branges spaces see \cite{dBr}.

A de~Branges space $\BB_E$ is called regular if
\begin{itemize}
\item[(i)] $\BB_E\subset\CC_a$ and
\item[(ii)] If $F\in \EE$ , then $(F(z)-F(w))/(z-w)\in \EE$ for any $w\in\C$.
\end{itemize}

The main example of a regular de Branges space is the Payley-Wiener space $\PWa$. In that case $E$ is the standard
exponential function $S^a=e^{iaz}$.

If $\Th$ is an inner function in the upper half-plane, we denote by $\KK_\Th$ the model space $H^2(\C_+)\ominus\Th H^2(\C_+)$.
Such spaces play a central role in the Nagy-Foias functional model theory, see \cite{Nik}.

Every de~Branges function $E(z)$ gives rise to an inner function $\Th(z):=E^\#(z)/E(z)$ and a model space $\KK_{\Th}$ that this inner function generates. There exists a well known isometric isomorphism between $\BB_E$ and $\KK_{\Th}$ given by $F\to F/E$.

Each inner function $\Th(z)$ determines a family of positive measures $\mu_{\alpha}$ on $\R$ indexed by $\abs{\alpha}=1$ in the following way
\begin{equation}\label{Herglotz}
\Re{\frac{\alpha+\Th(z)}{\alpha-\Th(z)}}=p_{\alpha}\Im{z}+\frac{1}{\Im{z}}\int\frac{d\mu_{\alpha}(t)}{\abs{t-z}^2},
\end{equation}
for some $p_{\alpha}\geq 0$. The number $p_{\alpha}$ can be viewed as a point mass at infinity for $\mu_{\alpha}$. Each measure $\mu_{\alpha}$ is singular, supported on the set $\{\Th=\alpha\}$, and is Poisson-finite. They are usually called the Clark measures for $\Th(z)$. Throughout the paper when we say that $\Th(z)$ corresponds to a measure $\mu$ on $\R$ we will always mean that $\mu$ is the Clark measure $\mu_1$ for $\Th$ and that $p_1=0$.
In this case, each function $f\in\KK_{\Th}$ can be represented by the Clark formula
$$f(z)=\frac{1-\Th(z)}{2\pi i}K(f\mu)(z),$$ where $K(f\mu)$ stands for the Cauchy integral $$K(f\mu)(z)=\int \frac{f(t)}{t-z}d\mu(t).$$ This formula gives an isometric isomorphism between $\KK_{\Th}$ and $L^2(\mu)$. For more on Clark measures see for instance \cite{PolSar}.

For those inner functions $\Th$ that can be represented as  $\Th(z)=E^\#(z)/E(z)$ for some de Branges function $E$, all  Clark measures $\mu_{\alpha}$ are discrete and their point masses can be computed by $\mu_{\alpha}(\lambda)=2\pi/\abs{\Th'(\lambda)}$ for $\lambda\in\{\Th=\alpha\}$.

We will call the measures $|E|^2\mu_{\alpha}$, where $\mu_\alpha$ is a Clark measure for $\Th(z)=E^\#(z)/E(z)$,  spectral measures of the corresponding de~Branges space. It is well known that for any spectral measure $\nu$ of a de Branges space $\BB_E$ the natural embedding gives
and isometric isomorphism between $\BB_E$ and $L^2(\nu)$.

On the real line each inner $\Th(z)$ coming from a de~Branges function can be written as $\Th(t)=e^{i\theta(t)},\ t\in\R$, where $\theta(t)$ is real analytic strictly increasing function. The function $\theta(t)$ is a continuous branch of the argument of $\Th(z)$ on $\R$. The phase function of the corresponding de~Branges space is defined by $\phi(t)=\theta(t)/2$ and is equal to $-\arg E$.

Next we give the definitions of the densities that will be used in our statements. Recall that the family of disjoint intervals $\{I_n\}$ on the real line (or the union of the intervals from that family, when it is convenient) is called long if $$\sum_n \left(\frac{|I_n|}{1+\text{dist}(0, I_n)}\right)^2=\infty.$$ Otherwise, we will say that the family is short. In the case when $\{I_n\}$ is short, $\bigcup I_n=\R$, and $\abs{I_n}\to \infty$ as $n\to \pm\infty$ we will call the family $\{I_n\}$ a short partition. 

In all our definitions and statements the intervals of a short partition, as well as discrete sequences of points, are assumed to be enumerated in
a natural increasing order.

\begin{definition} A sequence $\Lambda=\{\lambda_n\}$  with no finite accumulation point is said to be regular with density $a>0$ if there exists a short partition $\{I_n\}$ such that $$\#(\L\cap I_n)-a|I_n|=o(|I_n|) \text{ as } |n|\to \infty.$$
\end{definition}

\begin{remark} It can be shown that $\L$ is regular with density $a$ if and only if the following integral condition holds
$$\int \frac{\abs{n_\L(x)-ax}}{1+x^2}<\infty.$$ This definition of regularity was used in~\cite{Koo2} in the definition of Beurling-Malliavin densities.

\end{remark}


\begin{definition} The interior Beurling-Malliavin density $D_{BM}^-(X)$ of a closed set $X\subset \R$ is defined to be the supremum of all $a>0$ for which there exists a regular subsequence $\Lambda\subset X$ with density $a$. If no such subsequence exists $D_{BM}^-(X)=0$.

Similarly, for real sequences $\L$ one defines (the more famous) exterior density $D_{BM}^+(\L)$ as the infimum of $a$ such that
$\L$ is contained in a regular sequence with density $a$.
\end{definition}

The famous result of Beurling and Malliavin~\cite{BM} says that the zero set $\L$ of a function $F(z)\in\CC_{a\pi}$ must satisfy $D^+(\L)\leq a$. Conversely, for every sequence $\L$ with $D^+(\L)=a$ and every $k>a$, there exists non-zero $F(z)\in\CC_{k\pi}$ which vanishes at that sequence. However, $F(z)$ may have additional zeros.

In~\cite{PolGap} it was proved that a very similar statement holds for the zeros of functions in the Krein class. It turns out that in this case the regularity notion above needs to be refined to take into account a certain delicate separation condition. Namely, for any finite interval $I\subset \R$ and a Borel measure $\mu$ we define the energy of $\mu$, $E_I(\mu)$ by $$E_I(\mu)=\iint_{I, I} \log|x-y|d\mu(x)d\mu(y).$$ This is just the usual energy of the compactly supported measure $1_I(x)d\mu(x)$.

\begin{definition} We say that a real sequence $\Lambda=\{\lambda_n\}$  with no finite accumulation points is $a$-uniform if it is regular with density $a$ and satisfies the following energy condition: there exists a short partition $\{I_n\}$ such that

 $$ \sum_n\frac{\#(\L\cap I_n)^2\log_+|I_n|-E_{I_n}(dn_{\Lambda})}{1+\text{dist}(0,I_n)^2}<\infty.$$
\end{definition}

\begin{remark}

Notice that the series in the energy condition is positive.
The definition
given here is slightly different from the one given in \cite{PolGap, PolType}. Either one can be used in all of the results
and formulas from \cite{PolGap, PolType} or the present paper.
Note also that every separated regular sequence of density $a$ satisfies the energy condition automatically and therefore is $a$-uniform.

\end{remark}

The following technical lemma will be used several times in what follows.

\begin{lemma}\label{uff} If a discrete sequence $\L$ contains a $(d-\e)$-uniform subsequence for
every $\e>0$, then it also contains
a $d$-uniform subsequence. If $\L$ alternates between two sets $A$ and $B$, then this $d$-uniform
subsequence can be chosen to also
alternate between $A$ and $B$.
\end{lemma}

\begin{proof} Fix a small $\e>0$. Let $\L_1\subset\L$ be a $(d-\e)$-uniform sequence and let $\{I^{(1)}_n\}$ be the corresponding
short partition. Let $\L'_1\subset\L$ be a $(d-\e/2)$-uniform
sequence with corresponding short partition $\{J_n\}$.

We form a third short partition $\{K_n\}$ with the following properties:

\begin{itemize}

\item[(i)] $\{K_n\}$ is a superpartition of $\{J_n\}$, i.e. each interval $J_n$ is a subset of some interval $K_m$,

\item[(ii)] the intervals of $\{K_n\}$ are big enough, so that If two intervals $I^{(1)}_{n_k}$ and $I^{(1)}_{m_k}$
contain the endpoints of some interval $K_k$
then $|I^{(1)}_{n_k}|+|I^{(1)}_{m_k}|= o( |K_k|)$.
\end{itemize}

By lemma~4 in~\cite{PolGap} one can choose a subsequence $\Sigma$ of $\L'_1$
which is also $(d-\e/2)$-uniform with
a corresponding short partition $\{K_n\}$.

Let $\delta>0$. Choose $N$ big enough so that  the energy
sum for $\Sigma$ over $\{K_n\}_{|n|>N}$ is less than $\delta$, and so
that $I_{n_k}$ and $I_{m_k}$ for $k=\pm N$ are small in comparison
to $K_{\pm N}$.

We form a new sequence $\L_2$ in the following way. We keep $\L_1$ on all the
intervals $I^{(1)}_n$ that lie strictly between $K_{-N}$ and $K_N$.
On the rest of the line, keep $\Sigma$ on each $K_n$ with $|n|\geq N$.

We form a new short partition $\{I_n^{(2)}\}$  for $\L_2$ that consists of all intervals $I^{(1)}_n$ that lie between $K_{-N}$
and $K_N$, all $\{K_n\}_{|n|> N}$, instead of $K_N$ we include the interval $K_N\cup I^{(1)}_{n_N}$ and instead of
$K_{-N}$ we include $K_{-N}\cup I^{(1)}_{m_{-N}}$.

Then $\L_2$ is a $(d-\e/2)$-uniform sequence with a corresponding short partition $\{I_n^{(2)}\}$. The
corresponding energy sum  for $\L_2$ is
at most the energy sum for $\L_1$ plus $\delta$.

Reiterating this process we obtain sequences $\{\L_k\}$ that are $(d-\e/2^k)$-uniform with corresponding short partitions $\{I_n^{(k)}\}$ with energy sums at most the energy of $\L_1$ plus $1/2+\cdots +1/2^k$ (we choose $\delta = 2^{-i}$ in the $i$-th step). Using a diagonal process we
obtain a regular sequence of density $d$ whose energy sum converges, which is the desired sequence.

Finally, if $\L$ alternates between $A$ and $B$, in each step we can choose $\L_k$ to also alternate. The diagonal process will then also produce a sequence that alternates.

\end{proof}

\section{M.~Riesz-type criterion and its consequences}\label{Riesz}

For a given set of functions $\AA$ we will say that a positive measure $\mu$ is $\AA$-determinate if $\AA\subset L^1(\mu)$ and there is no other positive measure $\nu\neq \mu$, $\AA\subset L^1(\nu)$, such that
$$\int f(t)d\mu(t)=\int f(t) d\nu(t),$$ for all $f\in\AA$. Otherwise, we say that $\mu$ is $\AA$-indeterminate. We will study determinacy on linear spaces of entire functions $\EE$ satisfying the following properties:

\begin{itemize}
\item[(E1)] If $F(z)\in \EE$, then $F^\#(z):=\overline{F(\bar{z})}\in \EE$,
\item[(E2)] If $F(z)\in \EE$ and $w\in\C$, then $(F(z)-F(w))/(z-w)\in \EE$.
\item[(E3)] There exists $a\geq 0$ such that $\EE\subset \CC_a$.
\end{itemize}

There are many examples of sets $\EE$ that satisfy (E1)-(E3). The set of all polynomials $\PP$ is one such example. Other examples that we will use include the Bernstein class $\BB_a$, the Paley-Wiener space $\PP\WW_a$, and other regular de~Branges spaces (those de~Branges spaces which satisfy (E2) and (E3)).

It is not hard to see that a positive finite measure $\mu$ is $a$-determinate if and only if it is $\BB_a$-determinate or equivalently $\PP\WW_a$-determinate.

The next result gives a criterion for determinacy which generalizes the classical M. Riesz criterion for determinacy in the moment problem.
Our proof will also rely on the M. Riesz original idea. However, we will use  the de~Branges spaces as a substitute for the Gauss quadrature formula. We will use the following notation

$$\EE *\EE=\{F : F= GH, \text{ for some } G, H\in \EE\}.$$

\begin{theorem}\label{maj} Let $\EE$ be a set which satisfies (E1)-(E3) and let $\mu$ be a positive measure such that $\EE\subset L^1(\mu)\cap L^2(\mu)$. The following are equivalent:

\begin{itemize}
\item[(a)] $\mu$ is $\EE * \EE$-indeterminate
\item[(b)] The $L^2(\mu)$-closure of $\EE$ is a de~Branges space.
\item[(c)] The majorant  $m(x):=\sup\{\abs{F(x)} : F\in \EE, \norm{F}_{L^2(\mu)}\leq 1\}$ satisfies
$$\int \frac{\log{m(x)}}{1+x^2} dx < \infty,$$
\end{itemize}

\end{theorem}

As we already mentioned, a statement similar to the equivalence of (a) and (c) was recently presented in~\cite{Sod}. As a simple corollary we obtain the following result.

\begin{corollary}\label{Riesz1} A positive finite measure $\mu$ is $a$-determinate if and only if $$\int \frac{\log{m_{a/2}^{\mu}(x)}}{1+x^2} dx = \infty,$$ where $$m_{a/2}^{\mu}(w) := \sup \{ |F(w)| : F \in \spa\{e^{ixt}: t\in[-a/2,a/2]\}\text{ and } \|F\|_{L^2(\mu)} \leq 1 \}.$$
\end{corollary}

\begin{proof} First we prove that the integrability of  $\log{m(x)}/(1+x^2)$ implies that the $L^2(\mu)$-closure of $\EE$ is a de~Branges space. We will denote this closure by $\BB(\mu)$.

Each function $F$ in $\EE$ is in the Cartwright class $\CC_a$. Therefore $$\log|F(z)|\leq a|\Im z|+\frac{|\Im z|}{\pi}\int\frac{\log^+|F(t)|}{|t-z|^2}dt.$$ Taking a supremum over all $F(z)\in \EE$ with norm no greater than $1$ we first obtain
\begin{equation}\label{Cart}
\log|F(z)|\leq a|\Im z|+\frac{|\Im z|}{\pi}\int\frac{\log^+m(t)}{|t-z|^2}dt,
\end{equation}
for all non-real $z$. Using the simple estimate
$$\sup_{t\in\R}\abs{\frac{t-i}{t-z}}\leq \frac{1+\abs{z}}{\abs{\Im z}},$$ which is valid for all $z\notin \R$, we obtain
\begin{equation}
\log|F(z)|\leq a|\Im z|+C\frac{(1+\abs{z})^2}{\abs{\Im z}},
\end{equation}
where
$$C=\frac{1}{\pi}\int\frac{\log^+m(t)}{1+t^2}dt<\infty.$$
This immediately shows that $m(z)$ is bounded on compact sets that don't intersect the real axis. Standard application of the Levinson's log log theorem shows that $m(z)$ is also bounded on compact sets that intersect the real axis.

Let now $\{F_n\}$ be a sequence of functions in  $\EE$ which converges in $L^2(\mu)$ to some limit $f\in L^2(\mu)$. The goal is to show that $f$ is actually a restriction of some entire function $F\in  \BB(\mu)$. To see this notice that $\{F_n\}$ being Cauchy and the inequality  $$|F_k(z)-F_l(z)|\leq m(z)\|F_k-F_l\|_{L^2(\mu)},$$ together with the fact that $m(z)$ is locally bounded imply that $\{F_n\}$ converges uniformly on compact sets to some entire function $F$ which agrees with $f$ a.e. $\mu$. Since $F$ satisfies~\eqref{Cart}, $F(z)\in\CC_a$.

Let $F \in \BB(\mu)$ and let $F_n \in \EE$ be a sequence that converges to $F$ in $L^2(\mu)$.  Since $F_n(z)$ converges to $F(z)$ uniformly on compact sets, we have that for any $w \in \C$
\begin{equation}\label{exptype}
\frac{|F(w)|}{\|F\|} = \lim_{n \to \infty}\frac{|F_n(w)|}{\|F_n\|} \leq m(w).
\end{equation}
Therefore, non-real point evaluations are bounded in $\BB(\mu)$.  By considering $F_n^\#\in\EE$ we see that $F^{\#} \in \BB(\mu)$. By the axiomatic definition of de~Branges spaces, see \cite{dBr}, it remains to show that $F(z)(z-\bar{z}_0)/(z-z_0)\in \BB(\mu)$ for all $F\in \BB(\mu)$ and $z_0\notin\R$ such that $F(z_0)=0$. Since $$ F(z)\frac{z-\bar{z}_0}{z-z_0}=F(z)+(z_0-\bar{z}_0)\frac{F(z)}{z-z_0},$$ it is enough to show that $F(z)/(z-z_0)\in \BB(\mu)$, which again follows from the property that $(F_n(z)-F_n(z_0))/(z-z_0)\in\EE$. Therefore $\BB(\mu)$ is a de~Branges space.

If $\BB(\mu)$ is a de~Branges space pick any spectral measure $\nu$ of $\BB(\mu)$ which is different from $\mu$ (in case $\mu$ is one of them). Let $F\in\EE*\EE$ be arbitrary. Then by definition $F=GH$ for some $G, H\in \EE$. Since $\BB(\mu)$ is isomorphic to $L^2(\nu)$ and is isometrically contained in $L^2(\mu)$ we obtain that
$$\int F(t)d\mu(t)=\int G(t) \overline{H^\#(t)}d\mu(t)=\int F(t)d\nu(t).$$
Thus, $\mu$ is  $\EE*\EE$-indeterminate.

Next we show that if $\mu$ is $\EE*\EE$-indeterminate then $$\int \frac{\log{m(x)}}{1+x^2} dx < \infty.$$ If $\mu$ is $\EE$-indeterminate there exists another positive finite measure $\nu\neq\mu$ such that $\int F(t)d\mu(t)=\int F(t)d\nu(t)$ for all $F\in\EE*\EE$. This implies
\begin{equation}\label{eq1}
\int |F(t)|^2d\mu(t)=\int |F(t)|^2d\nu(t),
\end{equation} for all $F(z)\in \EE$.

Consider $\sigma=\mu-\nu$. We have that $$\int \frac{F(t)-F(z)}{t-z}d\sigma(t)=0,$$ for every  function $F(z)\in \EE$  and every $z\in \C$. This implies that $$F(z)=\frac{1}{G(z)}\int\frac{F(t)}{t-z}d\sigma(t),$$ where $G(z)=\int \frac{d\sigma(t)}{t-z}$. In particular, for $z=x+i$ with $x\in \R$ we have, $$|F(x+i)|\leq  \frac{1}{|G(x+i)|}\int |F(t)|d|\sigma|(t)\leq \frac{C}{|G(x+i)|}\left(\int|F(t)|^2d|\sigma |(t)\right)^{1/2}.$$ Now since $|\sigma| \leq \mu+\nu$ we obtain, $$|F(x+i)|\leq \frac{\sqrt{2}C}{|G(x+i)|}\left(\int|F(t)|^2d\mu(t)\right)^{1/2},$$ for every $F(z)\in \EE$. Consequently, $$m(x+i)\leq \frac{C'}{|G(x+i)|}.$$ The fact that $G(z+i)$, as a function of $z$, is analytic and bounded in $\C_+$ implies that $$\int \frac{\log{m(x+i)}}{1+x^2} dx \leq \int \frac{\log^-|G(x+i)|}{1+x^2}dx  + C''<\infty.$$

Now, let $F(z)$ be arbitrary element of $\EE$ with norm no greater than $1$. Then $F(z+i)$ as a function of $z$ is in the Cartwright class $\CC_{a}$. Therefore, for all $x\in \R$ we have $$\log|F(x)|\leq a+\frac{1}{\pi}\int\frac{\log^+|F(t+i)|}{(t-x)^2+1}dt.$$ Taking a supremum over all such $F(z)$ we first obtain $$\log|F(x)|\leq a+\frac{1}{\pi}\int\frac{\log^+|m(t+i)|}{(t-x)^2+1}dt,$$ and then again by taking a supremum on the left we deduce
$$\log m(x)\leq a+\frac{1}{\pi}\int\frac{\log^+|m(t+i)|}{(t-x)^2+1}dt.$$  Finally,
\begin{eqnarray*} \int \frac{\log{m(x)}}{1+x^2} dx &\leq& \frac{a\pi}{2}+\frac{1}{\pi}\int\int\frac{\log^+|m(t+i)|}{(t-x)^2+1}\frac{1}{1+x^2}dtdx \\ &=& a\pi+\int\frac{2\log m(x+i)}{4+x^2}dx<\infty.
\end{eqnarray*}
This completes the proof.

\end{proof}

Notice that the set of all polynomials $\PP$ satisfies $\PP*\PP=\PP$, so we derive the classical M.~Riesz criterion as a special case. In the case $\EE=\BB_a$ we obtain Corollary~\ref{Riesz1}.

As we mentioned in the introduction, the usefulness of the criterion above depends on the possibility to estimate the majorant $m(w)$. Below we consider some cases when this can be done.

\begin{corollary}\label{uniform} Let $\mu$ be a positive finite measure.  If there exists a non negative uniformly continuous function $w(t)$ on $\R$ satisfying $e^{w(t)}\in L^1(\mu)$ and $$\int\frac{w(t)}{1+t^2}dt=\infty,$$ then $\mu$ is $a$-determinate for every $a>0$.
\end{corollary}

\begin{remark} This lemma can be viewed as a stronger version of de Branges' gap theorem \cite{dBr}, which claims that if there is a function $w$ like in the lemma,
that satisfies $e^{w(t)}\in L^1(|\mu|)$ for a finite non-zero real measure $\mu$, then the support of $\hat\mu$ may not have any gaps. Our
statement can be equivalently reformulated as follows: if $e^{w(t)}\in L^1(\mu_+)$ then the support of $\hat\mu$ does not have any gaps.
\end{remark}

\begin{proof} First note that we may replace the uniform continuity assumption on $w(t)$ by the stronger one that $w(t)$ is uniformly Lipschitz (see~\cite[pg. 96]{Koo}). Next notice that every entire function $F(z)$ of exponential type at most $a$ for which $F(t)/e^{w(t)/2}\leq 1$ on $\R$ must also satisfy $\int |F(t)|^2d\mu(t)\leq 1$. Therefore, by combining the corollary in~\cite[pg. 236]{Koo} (for $W(t)=e^{w(t)/2}$) with Theorem~\ref{Riesz1}  we obtain $$\int \frac{\log{m_{a}^{\mu}(t)}}{1+t^2} dt = \infty,$$ for every $a>0$. The result now follows from Theorem~\ref{maj}. It should be noted that the idea used here is due to de~Branges.
\end{proof}

We next prove a strengthening of Levinson's gap result~\cite{Lev}.

\begin{corollary} Let $\mu$ be a positive finite measure.  If there exists a non negative function $w(t)$ on $\R$ which is increasing on $[0,\infty)$ and satisfies $e^{w(t)}\in L^1(\mu)$ and $$\int_1^{\infty}\frac{w(t)}{1+t^2}dt=\infty,$$ then $\mu$ is $a$-determinate for every $a>0$.
\end{corollary}

\begin{remark}
Equivalently, if $w$ is like in the last corollary and $\mu$ is a real non-zero finite measure such that $e^{w(t)}\in L^1(\mu_+)$ then $\supp\hat\mu$ has no gaps. The original theorem by Levinson has a stronger requirement that $e^{w(t)}\in L^1(|\mu|)$
\end{remark}

\begin{proof} Let $v(t)$ be the largest minorant of $w(t)$ which is uniformly Lipschitz with Lipschitz constant $1$. Then (see the lemma in~\cite[pg. 239]{Koo})  $$\int \frac{v(t)}{1+t^2}dt=\infty,$$ and the result follows from the previous corollary.
\end{proof}

Finally, we prove a strengthening of the Beurling's gap theorem as discussed in the introduction.

\gap*



\begin{proof} Define $w(t)$ as zero outside of the intervals $(a_n,b_n)$ and
  as $\textrm{dist} (t, \R\setminus (a_n,b_n))$ on $(a_n,b_n)$. This function is clearly non negative, uniformly continuous and satisfies $e^{w(t)}\in L^1(\mu)$. Moreover, the assumption that $\cup (a_n,b_n)$ is long yields that $$\int\frac{w(t)}{1+t^2}dt=\infty.$$ The result now follows from Corollary~\ref{uniform}.
\end{proof}

\section{Extreme measures in the indeterminate case}\label{extreme}
Let $\AA$ be a set of functions on $\R$.
For a given positive measure $\mu$ with $\AA\subset L^1(\mu)$ denote by $\MM_{\AA}(\mu)$ the set of all finite positive measures $\nu$ such that
$\AA\subset L^1(\nu)$ and $$\int F(t)d\mu(t)=\int F(t) d\nu(t)$$ for all $F\in \AA$. Clearly, this set is convex and therefore it is interesting to describe its extreme points (if they exist). The following result, which in essence goes back to M. Naimark, gives a description of the extreme points of $\MM_{\AA}(\mu)$.

\begin{theorem} A positive finite measure $\nu\in \MM_{\AA}(\mu)$ is an extreme point of $\MM_{\AA}(\mu)$ if and only if the span of the set $\AA$ is dense in  $L^1(\nu)$.
\end{theorem}

\begin{proof} Assume that  $\nu\in \MM_{\AA}(\mu)$ is not an extreme point of $\MM_{\AA}(\mu)$. Then $\nu=\alpha\nu_1+(1-\alpha)\nu_2$ for some $\nu_1, \nu_2\in \MM_{\AA}(\mu)$ and $0<\alpha<1$. It is easy to see that
$$\phi(f)=\int f(t)d\nu(t)-\int f(t)d\nu_1(t),$$ is a non trivial bounded linear functional on $L^1(\nu)$ which vanishes on the set $\AA$. Therefore, the span of $\AA$ cannot be dense in  $L^1(\nu)$.

Conversely, assume that the span of $\AA$ is not dense in  $L^1(\nu)$. Then there exists a nontrivial bounded linear functional $\phi$ on $L^1(\nu)$ which vanishes identically on $\AA$. We can assume that its norm is $1$. In this case both $\phi_1(f)=\int f(t)d\nu(t)-\phi(f)$ and $\phi_2(f)=\int f(t)d\nu(t)+\phi(f)$ are nonnegative linear functionals on $C_0(\R)$ in the sense that both $\phi_1(f)\geq 0$ and $\phi_2(f)\geq 0$ whenever $f(t)\geq 0$ on $\R$.  Therefore, there exist positive finite measures $\nu_1$ and $\nu_2$ such that $\phi_1(f)=\int f(t)d\nu_1(t)$ and  $\phi_2(f)=\int f(t)d\nu_2(t)$ for all $f\in L^1(\nu)$. Notice that $\phi\neq 0$ implies that $\phi_1\neq \phi \neq \phi_2$ and hence $\nu_1\neq \nu \neq \nu_2$.  Also for all $F\in\AA$ we have that $\int F(t)d\nu_1(t)=\int F(t)d\nu(t)=\int F(t)d\nu_2(t)$ and therefore $\nu_1, \nu_2 \in \MM_{\AA}(\mu)$. Finally, since clearly $$\nu=\frac{1}{2}\nu_1+\frac{1}{2}\nu_2$$ we obtain that $\nu$ cannot be an extreme point for $\MM_{\AA}(\mu)$.
\end{proof}

We will call a lower semi-continuous function $W:\R\to [1,\infty)$ a weight on $\R$. We will say that a measure $\mu$ on $\R$ is $W$-finite
if
$$||\mu||_W=\int Wd|\mu|<\infty.$$

Let  $\BB(E)$ be a de Branges space. We will call  a weight $W$ an $E$-weight if
 $F(x)/W(x)=o(1)$ as $x\to\pm\infty$ for any $F\in \BB(E)$.

Such weights exist for any de Branges space and can, in most cases, be chosen in a natural way.
For instance, for the generating Hermit-Biehler function $E$ one may consider entire functions $A=\frac 12(E^\#+E)$ and $B=\frac 1{2i}(E^\#-E)$ and
define  the weight $$W_E(x)=1+(B'(x)A(x)-B(x)A'(x))^{1/2}.$$ Since the reproducing kernels of $\BB(E)$ satisfy
$$||K^E_x||_{\BB(E)}=(B'(x)A(x)-B(x)A'(x))^{1/2}$$
for any $x\in\R$,
any weight $W$, such that $W_E=o(W)$ near infinity, is an $E$-weight.
 Note that $W_E$ itself is not always an $E$-weight: consider for instance the case when $E$ is
a polynomial. However, in the most important example for our purposes, $\BB(E)=\PWa$, $W_E=1+a/\pi$ is an $E$-weight.

If $W$ is an $E$-weight then $\BB(E)\subset L^1(|\mu|)$ for any $W$-finite measure $\mu$.
We say that a measure $\mu$ annihilates  $\BB(E)$ if it is $W$-finite for some $E$-weight $W$ and satisfies
$$\int Fd\mu=0$$
for all $F\in\BB(E)$.

Notice that positive and negative parts of a $\BB_E$-annihilating measure are $\BB_E$-indeterminate.
It will be crucial for us to examine the properties of the extreme points of the set of all  annihilating measures of a given de~Branges space $\BB_E$. Our approach is based on de~Branges' extreme point method~\cite{dBrSW, dBrA, dBrB, dBr}.

Let $A$ and $B$ be two disjoint closed subsets of $\R$ and let $W$ be an $E$-weight.
Denote by $\MM_{E}^W(A,B)$ the set of all real measures $\sigma$, $\|\sigma\|_W\leq 1$, which annihilate $\BB_E$, and such that $\supp\mup\subset A$, $\supp\mum\subset B$.  Here, as usual, $\mup$ and $\mum$ denote the positive and negative parts of $\sigma$ in the canonical Jordan decomposition $\sigma=\mup-\mum$. Although, in view of the problem of $a$-determinacy, the most important particular case for us is when
$E=\exp{(iaz)}$, $\BB_E=\PWa$ and $W=W_E=1+a/\pi$, in this section we study the general case.

In the rest of the statements of this section we assume that $\BB_E\neq \{0\}$ is a non-trivial de Branges space and $W$ is an $E$-weight.


\begin{lemma}\label{KrMil} The set $\MM_E^W(A, B)$ is convex and compact with respect to the  weak-* topology of the set of all $W$-finite measures on $\R$. In particular, if it contains a non-zero measure then it also contains non-zero extreme points.
\end{lemma}

\begin{proof} Convexity part is obvious. To show that  $\MM^W_E(A,B)$ is weak-* compact it is enough to show that it is weak-* closed. Let $\nu_n\in \MM^W_E(A,B)$ be a sequence of measures that converges to some measure $\nu$ in the weak-* sense. Notice that the predual space of the space of all $W$-finite measures is $C_0(W)$, the space of all
continuous functions $f$ on $\R$ satisfying $f=o(W)$ at infinity. Therefore, $\BB_E$ is a subspace of the predual space $C_0(W)$ and hence $\nu$ must also annihilate $\BB_E$.

It remains to show that $\supp\nup\subset A$ and $\supp\num\subset B$. Suppose that $\supp\nup\nsubseteq A$. Then there exists an open set $S$  disjoint from $A$ such that $\nup(S)>0$. Notice that if $S\cap \supp\num\neq \emptyset$ then $\num(S\cap \supp\num)>0$ in which case $\nup(S\cap \supp\num)=0$. Therefore, without loss of generality we can assume that $\num(S)=0$ because otherwise we can replace $S$ with $S\cap (\supp\num)^c$. Inner regularity of $\nup$ implies that there exists a compact set $K\subset S$ with $\nup(K)>\nup(S)/2$. Similarly, outer regularity of $\num$ implies existence of an open set $G$ such that  $S\subset G\subset A^c$ and $\num(G)<\num(S)+\nup(S)/4=\nup(S)/4$. By Urysohn's lemma there exists a continuous function $f$ which vanishes at infinity and such that $f=1$ on $K$, $f=0$ outside of $G$ and $0\leq f\leq 1$ everywhere. Clearly, $f\in C_0(W)$. Now, it is easy to see that $$\int fd\nu\geq \nup(K)-\num(G)\geq \nup(S)/4>0.$$ On the other hand, for all $n\in\N$ we have $\int f d\nu_n=\int f d\nu^-_n \leq 0$. We have a contradiction. The proof that $\supp\num\subset B$ is similar.

The Krein-Milman theorem now implies that $\MM^W_E(A,B)$ is a closed convex hull of its extreme points, and clearly some of them must be non-zero since $\MM^W_E(A,B)$ contains a non-zero element.

\end{proof}

\begin{lemma}\label{extremel} Let $\nu$ be an extreme point in $\MM^W_E(A,B)$. Then any function $f\in L^1(|\nu|)$ which is not in the $L^1(|\nu|)$-closure of $\BB_E$ must satisfy $$\int f(t)d\nu(t)\neq 0.$$
Moreover, the space of all bounded linear functionals on $L^1(|\nu|)$ which vanish on $\BB_E$ is one-dimensional.
\end{lemma}

\begin{proof} Let $f\in L^1(|\nu|)$ be a non-zero function which is not in the $L^1(|\nu|)$-closure of $\BB_E$. Then there exists a bounded linear functional which is zero on $\BB_E$ and non-zero at $f$. Each such functional $L$ on $L^1(|\nu|)$ can be represented as $L(f)=\int f(t)h(t)d|\nu|(t)$ for some $h\in L^{\infty}(|\nu|)$. Therefore, we have $\int F(t)h(t)d|\nu|(t)=0,$ for every $F\in \BB_E$ and  $\int f(t)h(t)d|\nu|(t)\neq 0$. Let $d\nu(t)=b(t)d|\nu|(t)$ be the polar decomposition of $\nu$. Using the fact that $F\in \BB_E$ implies $F^{\#}\in\BB_E$ we can assume with no loss in generality that $h/b$ is real valued. By adding a constant if necessary, we may assume that $h/b$ is a nonnegative function. Multiplying
 by a constant if necessary, we can assume that $ h/b$ is bounded by $1$ on the real line. Choose $\l\in (0,1)$ and put  $\nu_1= hd|\nu|$, $\nu_2=(\nu-\l hd|\nu|)/(1-\l)$. Both $\nu_1$ and $\nu_2$ are members of $\MM^W_E(A,B)$ and $\nu=\l \nu_1+(1-\l)\nu_2$. Since $\nu$ is an extreme point and clearly $\nu_1\neq 0$ we must have that $\nu_2=0$. Therefore, $\nu=\l h |\nu|$ and hence $$\int f(t)d\nu(t)\neq 0.$$
For the second part, notice that  $\nu=\l h |\nu|$ implies $\l h=b$ and therefore $L(f)=\l\int f(t)d\nu(t)$.
\end{proof}

\begin{lemma}\label{close} Let  $\BB_E$ be a regular de~Branges space such that $\BB_E\subset \CC_a$ and let $\mu$ be a finite positive measure
such that $\BB_E$ is not dense in $L^1(\mu)$. Then for every function $f\in L^1(\mu)$ which is in the $L^1(\mu)$-closure of $\BB_E$ there exists an entire function $F$ in the Cartwright class $\CC_a$ such that $F=f$, $\mu$-a.e.
\end{lemma}
\begin{proof} The proof is very similar to the proof of Theorem~\ref{maj}. Since  $\BB_E$ is not dense in $L^1(\mu)$, there exists
$h\in L^\infty(\mu)$ such that $\int hf\mu=0$ for all $f$ in the $L^1(\mu)$-closure of $\BB_E$. Put $\nu=h\mu$.
Let $$m(w) = \sup \{ |F(w)| : F \in \BB_E \text{ and } \|F\|_{L^1(|\nu|)} \leq 1 \}.$$ Using the fact that $\nu$ annihilates $\BB_E$ it follows that $$\int{\frac{F(t)-F(z)}{t-z}}d\nu(t)=0,$$ for all $F\in \BB_E$ and non-real $z$.  Then clearly $$m(z)\abs{\int \frac{d\nu(t)}{t-z}}\leq \frac{C}{|\Im{z}|}.$$ Exactly as in the proof of Theorem~\ref{maj} this implies that $m(t)$ is Poisson summable, that $m(w)$ is locally bounded, and that $$\log m(w)\leq a|\Im w|+\frac{|\Im w|}{\pi}\int\frac{\log^+m(t)}{|t-w|^2}dt,$$ for all non-real $w$.  Let now $\{F_n(z)\}$ be a sequence of functions in  $\BB_E$ which converges in $L^1(\mu)$ to some limit $f\in L^1(|\nu|)$. Then $\{F_n\}$ being Cauchy and the inequality  $$|F_k(z)-F_l(z)|\leq m(z)\|F_k-F_l\|_{L^1(|\nu|)},$$ together with the fact that $m(z)$ is locally bounded imply that $\{F_n\}$ converges uniformly on compact sets to some entire function $F$. On the other hand, $F_n$ converge to $f$ in $L^1(\mu)$. Therefore, $F$ agrees with $f$, $\mu$-a.e. Finally, it is then easy to see that $F\in \CC_a$.
\end{proof}



\begin{lemma}\label{Krein} Let $\BB_E\subset \CC_a$ be a regular de~Branges space. If $\nu\in\MM^W_E(A,B)$ is a non-zero extreme point then it is supported
on a discrete sequence and its Cauchy integral $K\nu:\C\to\hat\C$ is a non-vanishing function.
 Moreover, the supports of $\nup$ and $\num$ are interlaced, i.e., for any two points $a,b$ in the support of $\nup$ there exists a point $c$ from the support of $\num$ such that $a<c<b$ and vice versa.
\end{lemma}

\begin{proof}
Since $\BB_E$ is a regular space, it has no common zeros and the support of $\nu$ contains more than one point. Let $(a,b)$ be any finite interval that contains at least two points of the support of $\nu$. There exists a function $f\in L^1(|\nu|)$ which vanishes outside of $(a,b)$ such that $\int |f(t)|d|\nu|(t)=1$ and $\int f(t)d\nu(t)=0$. By Lemma~\ref{extremel} and Lemma~\ref{close}, there exists an entire function $F\in \CC_a$ such that $F=f$ a.e. $|\nu|$. It follows that the support of $\nu$ outside of $(a,b)$ must be contained in the zeros of $F$. Since $F$ is non zero, the support of $\nu$ outside of $(a,b)$ cannot have a finite accumulation point. Finally, choosing a new interval $(a,b)$ to be disjoint from the old one and repeating the argument we can show the whole support of $\nu$ must be a discrete set.

The statement that  $K\nu$ has no zeros is now equivalent to the statement that
$G=\frac{1}{K\nu}$ is an entire function, see the proof of Theorem~66 in~\cite{dBr}.
(See also Lemma 4 in \cite{PolType}.)
To prove the interlacing of $\supp\nu_\pm$ notice that the function $G(z)$ must vanish at every point $c$ in the support of $\nu$, and for all such $c$,  $G'(c)=1/\nu(c)$. Since $G$ is a real analytic function on $\R$ with simple zeros, for any two of its consecutive zeros $a$ and $b$ we must have $G'(a)G'(b)<0$. Therefore, $\nu(a)\nu(b)<0$ for any two consecutive points in the support of $\nu$.
\end{proof}

Let $\Th=E^\#/E$ be the inner function that corresponds to the de~Branges function $E$. If $\nu$ annihilates $\BB_E$, then $E\nu$ annihilates $K_\Th$ and hence the Cauchy integral of $\bar{E}\nu$  is divisible by $\Th$,
see \cite{Aleksandrov} or \cite{PolGap}.
In particular, if $\nu\in\MM^W_E(A,B)$ is an extreme point, then, by Lemma~\ref{extremel}, the only (up to a constant) bounded linear functional on $L^1(|\nu|)$ which annihilates the model space $\KK_{\Th}$ is given by  $L(f)=\int f(t)E(t)d\nu(t).$ Therefore, there is no non-constant bounded function $h$ for which the Cauchy integral $K(h\bar{E}\nu)$ is divisible by $\Th$. We will use this observation in the proof of our next result.

For a discrete sequence $\L=\{\lan\}$ we denote by $n_\L$ its counting function, that is defined to be $0$ at $0$, is constant on each
interval $(\lan,\l_{n+1}]$ and jumps up by one at each $\lan$.

\begin{theorem}\label{poissonsummable} Let $\BB_E$ be a regular de~Branges space and let $\phi(t)$ be the corresponding phase function.
Let $W$ be an $E$-weight and let $\mu$ be a $W$-finite measure. Suppose that there
 exists another $W$-finite measure $\nu$ such that
 $$\int Fd\mu=\int Fd\nu$$
 for all $F\in \BB(E)$. Then the support of $\mu$ contains a sequence $\L$ whose counting function satisfies $2\pi n_{\L}(t)-\phi(t)=\tilde{h}(t)$ for some $h\in L^1(\Pi)$.
\end{theorem}

\begin{proof}  Notice that $\sigma=\mu-\nu$ is a finite signed measure which annihilates $\BB_E$ and (after scaling) $\sigma\in\MM_E^W(A,B)$ for $A:=\supp\sigma^+, B=\supp\sigma^-$. Notice that clearly $A\subset \supp\mu$ and $B\subset \supp\nu$. By Lemma~\ref{KrMil} there exists an extreme measure $\eta\in \MM^W_E(A,B)$ which has to be discrete by Lemma~\ref{Krein}.

Let $\Th=E^\#/E$ be the inner function that corresponds to the de~Branges function $E$. Then $\Th(t)=e^{2i\phi(t)}$. Let $\Psi$ be the inner function corresponding to the Clark measure $\abs{E}|\eta|$. The goal is to show that there exists an outer function $h\in \ker T_{\bar{\Psi}\Th}$ such that $\Psi\bar{\Th}\bar{h}=h$ is also outer.

Define $g=(1-\Psi)K(\bar{E}\nu)$. Then clearly $\bar{E}\eta=g|E||\eta|$ and $g\in \KK_{\Psi}$. Moreover, the fact that $\eta$ is real and annihilates $\BB_E$ implies that the Cauchy integral $K(\bar{E}\eta)$ is divisible by $\Th$. Therefore, $g=\Th h$ for some $h\in\HH^2\cap\HH^\infty$. Assume that $h$ is not outer, i. e., there exists an inner function $\Phi$ such that $h=\Phi k$ for some $k\in\HH^2\cap \HH^\infty$. Then $\Th(1+\Phi) k\in \KK_{\Psi}$ and hence, by Clark's formula $$\Th(1+\Phi) k=(1-\Psi)K(\Th(1+\Phi)k|E||\eta|).$$ So, the Cauchy integral of the measure $\Th(1+\Phi)k|E||\eta|$ is also divisible by $\Th$. However,
$$\Th(1+\Phi)k|E||\eta|=(1+\Phi)\frac g\Phi|E||\eta|=\frac{1+\Phi}{\Phi}\bar{E}\eta.$$ Thus, we obtained a non-constant bounded function $(1+\Phi)\bar{\Phi}$ such that the Cauchy integral of $(1+\Phi)\bar{\Phi}\bar{E}\eta$ is also divisible by $\Th$. This is in contradiction with the fact that $\eta$ is an extreme point, see Lemma \ref{extremel} and the comment before the statement of the theorem.  Therefore, $h$ is outer.

To show that $\Psi\bar{\Th}\bar{h}=h$ is simple. Just notice that $$\Psi\bar{\Th}\bar{h}=\Psi\bar{g}=(1-\Psi)K(\bar{\Th}\bar{E}\eta)=(1-\Psi)K(\bar{\Th}g|E||\eta|)=(1-\Psi)K(h|E||\eta|)=h.$$ Therefore, $$e^{i(\psi(t)-2\phi(t))}=\bar{\Psi}(t)\Th(t)=\frac{\bar{h(t)}}{h(t)}=\frac{e^{h(t)-i\tilde{h}(t)}}{e^{h(t)+i\tilde{h}(t)}}=e^{-2i\tilde{h}(t)},$$ where $\psi(t)$ denotes the increasing argument of the inner function $\Psi$. Thus, $\psi(t)-2\phi(t)\in \tilde{L^1}(\Pi)$.

Let $\Lambda:=\supp \eta_+\subset A$. Recall that $\supp\eta=\{\Psi=1\}$. By Lemma \ref{Krein}, the supports of $\eta_\pm$ are interlaced and hence $\abs{4\pi n_{\L}(t)-\psi(t)}\leq 4\pi$. So, we also have that $4\pi n_{\L}(t)-\psi(t)\in \tilde{L^1}(\Pi)$ and we get the desired conclusion $$2\pi n_{\L}(t)-\phi(t)\in \tilde{L^1}\left(\Pi\right).$$

\end{proof}

The last statement can also be formulated  as follows:

\begin{corollary}\label{poissonsummablecor} Let $\BB_E$ be a regular de~Branges space and  let $\phi(t)$ be the corresponding phase function. If $\mu$ is a  non-zero measure that annihilates $\BB_E$ then the support of $\mu_+$ ($\mu_-$) contains a sequence $\L$ whose counting function satisfies $2\pi n_{\L}(t)-\phi(t)=\tilde{h}(t)$ for some $h\in L^1(\Pi)$.
\end{corollary}

In the case $\BB_E=\PWa$ we have a more precise result in the form of Theorem \ref{indeterminate} that will be proved below.

\section{Measures annihilating $\PWa$.}

Note that in the case $E=S^a$ we have $W_E=\const$.
Since $F(x)=o(1)$ as $x\to\pm\infty$ for all $F\in\PWa$, $W=1$ is an $E$-weight that can be used in all statements from the last section.
In what follows we write $\MM_a(A,B)$ instead of $\MM^1_{PW_a}(A,B)$, i.e. $\MM_a(A,B)$ is the set of real finite measures $\mu$ of norm at most 1
that annihilate $\PWa$ and satisfy $\supp\mu_+\subset A,\supp\mu_-\subset B$.

\indeterminate*

\begin{proof} If $\mu$ is $a$-indeterminate, then there exists a finite positive measure $\nu\neq\mu$ such that $\mu-\nu$
annihilates $\PWa$. Put $A=\supp\mu$, $B=\supp\nu$ and consider $\MM_a(A,B)$. Let $\eta$ be a non-trivial extreme point of
$\MM_a(A,B)$. Let $\L=\{\lan\}=\supp \eta$.
Since $\eta$ is extreme, it is a zero set of an entire function from the Krein class and therefore $\L$ must be
a regular sequence with density $a/\pi$.
By  lemma \ref{Krein}, $\L'=\{\l_{2n}\}\subset \supp\mu$.

The main theorem of \cite{PolGap} implies that $\L$ contains
a $(\frac{a}{\pi}-\e)$-uniform subsequence for any $\e>0$. It follows that $\L'$ contains a $(\frac{a}{2\pi}-\e)$-uniform subsequence for any $\e>0$.
Using lemma~\ref{uff}, one can select a  $\frac{a}{2\pi}$-uniform subsequence of $\L'$.
 \end{proof}

The following result represents a refinement of de~Branges Theorem~66 in~\cite {dBr}.

\begin{theorem} Let $A$ and $B$ be disjoint closed subsets of $\R$. The necessary and sufficient condition for $ \MM_a(A,B)$ to contain a non-zero measure is that there exists a sequence $\Lambda=\{\lambda_n\}$, satisfying  $ \{\lambda_{2n}\}\subset A, \{\lambda_{2n+1}\}\subset B$, which is the zero set of an entire function in the Krein class $\KK_a$.
\end{theorem}

\begin{proof}  Assume that $\MM_a(A,B)$ contains a non-zero measure. It follows from Lemma~\ref{Krein} that  there exists an entire function $G(z)$ in the Krein class $\KK_a$ with zero set $\Lambda=\{\lambda_n\}$ satisfying  $ \{\lambda_{2n}\}\subset A, \{\lambda_{2n+1}\}\subset B$. Conversely, if such function exists define a signed discrete measure $\sigma$ supported on $\L$ simply by  $\sigma(\{\lambda_n\})=1/G'(\lambda_n)$. Using the properties of $G(z)$ it is easy to show that this measure (or perhaps $-\sigma$) belongs in $\MM_a(A,B)$.
\end{proof}

To prove our oscillation result below we will need the following technical lemma. It basically says that any subsequence of real zeros of a given function $F\in\PWa$ can be replaced by double zeros without changing the exponential type and still keeping the boundedness on the real line.

\begin{lemma} Let $F\in\PWa$ with infinitely many real zeros and let $\Lambda=\{\lan\}$ be a subsequence of the real zeros of $F$ indexed so that
$$\dots<\l_{-2}<\l_{-1}<0<\l_1<\l_2<\dots $$
Assume that $\sup_n \l_{n+1} -\l_n<\infty$.
Put $\gamma_n=(\l_{2n-1}+\l_{2n})/2$ for $n$ positive and $\gamma_n = (\l_{2n}+\l_{2n+1})/2$ for $n$ negative.

Then the entire function $G(z)$ defined by
$$G(z)=\frac{\prod_{\gamma\in\Gamma} \left(z-\gamma\right)^2}{\prod_{\l\in\L} \left(z-\l\right)}F(z)$$
is bounded on $\R$.

\end{lemma}

\begin{proof} Let us assume for the time being that the products in the definition of $G$ converge (we address convergence at the end of the proof).
 Define $u(t)=2n_{\G}(t)-n_{\L}(t)$.
On the real line we have
$$\log\abs{G(t)}=-\ti u(t) +\log\abs{F(t)}+\const.$$
Denote $I_n=(\l_{2n-1},\l_{2n})$ for $n$ positive and $I_n = (\l_{2n},\l_{2n+1})$ for $n$ negative.
Let $x$ be a real point.
Since $u$ forms an atom on each $I_n$, its harmonic conjugate can be calculated as
$$\ti u(x)=\frac 2\pi \int\frac{u(t)}{x-t}dt+\const,$$
i.e. the integral on the right converges.
 Let $U$ be the union of $I_n$ that intersect the interval $(x-1,x+1)$.
 Then for any $I_n\not\in U$, $n>0$,
 $$\left|\int_{I_n}\frac{u(t)}{x-t}dt\right|=\left|\ln\frac{(\l_{2n-1}-x)(\l_{2n}-x)}{(\gamma_n-x)^2}\right|=$$
 $$\left|\ln\left(1-\frac{(\l_{2n-1}-\l_{2n})^2}{4(\gamma_n-x)^2}\right)\right|<D |I_n|^2 dist^{-2}(x,I_n)$$
 for some positive constant $D$ depending only on $C=\sup_n\abs{\l_{n+1} -\l_n}$.
 A similar estimate can be written for $n<0$. Hence
$$\left|\int_{\R\setminus U}\frac{u(t)}{x-t}dt\right|=\left|\sum_{I_n\not\in U}\int_{I_n}\frac{u(t)}{x-t}dt\right|
<C_1.$$
Suppose that $U=I_n\cup I_{n+1}\cup ...\cup I_{N}$ for some $0<n\leq N$. Let $x>0$, $x\in [\gamma_k,\gamma_{k+1}]$, $n\leq k\leq N$.
 Notice that $\L\cap [\gamma_k,\gamma_{k+1}]=\{\l_{2k},\l_{2k+1}\}$ and
$$F(x)/(x-\l_{2k})(x-\l_{2k+1})<C_2$$
because $F(x)$ is a $\PWa$ function and $\l_{2k},\l_{2k+1}$ are its zeros.
Also notice that
$$\abs{\frac{\prod_{n\leq m\leq N} \left(x-\gamma_m\right)^2}{\prod_{2n-1\leq m\leq 2N,\ m\neq 2k, 2k+1} \left(x-\l_m\right)}}<(1+C)^2.$$
Altogether we obtain
$$ \abs{G(x)}=\abs{\frac{\prod_{\gamma\in\Gamma} \left(x-\gamma\right)^2}{\prod_{\l\in\L} \left(x-\l\right)}F(z)}=
$$$$\abs{
\exp{\left[-\frac 2\pi \int_{\R\setminus U}\frac{u(t)}{x-t}dt\right]}\frac{\prod_{n\leq m\leq N} \left(x-\gamma_m\right)^2}{\prod_{2n-1\leq m\leq 2N} \left(x-\l_m\right)}F(z)}\leq e^{C_1}C_2(1+C)^2<\const.$$

Note that our estimates of $\ti u$ show that the products in the definition of $G$ converge for $z\in\R\setminus\L$. Via  standard argument,
convergence on the line implies convergence on $\C\setminus \L$.

\end{proof}

Now we can prove our main result stated in the introduction. Recall that for a pair of closed sets of real numbers $A$ and $B$ the
gap characteristic is defined as
$$G(A,B)=\sup\{a>0: \MM_a(A,B)\neq\{0\}\}.$$

\oscillation*

\begin{proof} Assume that $\MM_a(A,B)$ contains a non-zero measure. By the previous theorem, it then must contain a non-zero extreme measure $\nu$ with discrete support $\L=\{\lan\}$ such that $\{\lambda_{2n}\}\subset A, \{\lambda_{2n+1}\}\subset B$.
By the main theorem from \cite{PolGap}, $\L$ contains a $(a/\pi-\e)$-uniform subsequence $\L'$. One can then choose a $(a/\pi-2\e)$-uniform subsequence $\L''$ of $\L'$ that alternates properly between the sets $A$ and $B$.
Hence we are done with  one of the inclusions.

In the opposite direction, let $\Lambda=\{\lan\}$ be a $d$-uniform sequence with $\{\lambda_{2n}\}\subset A, \{\lambda_{2n+1}\}\subset B$. Let $\e>0$ be arbitrary. We will construct a finite signed measure $\sigma \in \MM_{\pi d-2\epsilon}(A,B)$ which will finish the proof. From~\cite{PolGap} we know first that there exists a finite signed measure $\sigma_1$ supported on $\Lambda$ with a spectral gap of the size $2\pi d-\e$. Using the extreme point procedure from above if needed we can assume that this measure $\sigma_1$ oscillates between the consecutive points in its support. More precisely, if we denote the support of $\sigma_1$ by $\L':=\{\l'_n\}\subset \Lambda$, we have $\sigma_1(\l'_n)\sigma_1(\l'_{n+1})<0$ for all $n$. This measure however may not be in $\MM_{\pi d-\epsilon}(A,B)$ since we do not know that $\L'$ alternates between the sets $A$ and $B$.

Consider the intervals $(\l'_n, \l'_{n+1})$ which contain odd number of points from $\Lambda$. These are exactly the bad intervals since in this case $\l'_n$ and $\l'_{n+1}$ are both in $A$ or both in $B$ even though $\sigma_1$ has opposite sign at these points. To fix this form a sequence $\Gamma$ by picking one point from $\L$ in each such interval. Then $D_{BM}^+(\Gamma)\leq D_{BM}^+(\L\setminus \L')\leq \e$. By the Beurling-Malliavin theorem there exists an entire function $F(z)\in \PP\WW_{\e}$ which vanishes at $\Gamma$. We may assume that $F$ is real on the real line (otherwise take $F+F^\#$).
If the real zero set of $F(z)$ is exactly $\Gamma$, then we will be done by taking $d\sigma:=Fd\sigma_1$. However, $F(z)$ may have additional real zeros. Denote by $\Gamma'$ these zeroes. Without loss of generality we can assume that $\Gamma'$ is an infinite sequence with bounded gaps (this condition will be needed to apply the last lemma).
Otherwise, we can add an arithmetic progression of density $\epsilon$ to the zero set of $F(z)$.

Form the sequence $\Delta=\{\delta_n\}$ of the arithmetic means of consecutive pairs in $\Gamma'$, like in the last lemma.
We can now apply the last lemma to replace the unwanted real zeros of $F(z)$ with double zeros and obtain an entire function $G(z)\in\BB_{\e}$ such that $G$ is real and bounded on $\R$,
the zero set of $G$ on $\R$ is exactly  $(\Gamma\setminus\Gamma')\cup \Delta$ and $G$ has double zeros on $\Delta$. 
Consider the measure $d\sigma=Gd\sigma_1$. Since $G$ is bounded and real on $\R$ this is a finite signed measure. The fact that $G$ is of exponential type no greater than $\e$ implies that $\sigma$ has a spectral gap of the size $2\pi d-3\e$. In addition, in each of the bad intervals $(\l'_n,\l'_{n+1})$ the function $G$ has exactly one simple zero and the rest of its zeros are double zeros. Therefore, $\sigma$ has opposite sign point masses at $\l'_n$ and $\l'_{n+1}$ and has the same good behavior as $\sigma_1$ on the rest of $\L'$. Thus, $\sigma\in \MM_{\pi d-2\epsilon}(A,B)$.
\end{proof}

As a consequence of the oscillation theorem we derive the following result stated in the introduction.

\DetGap*


\begin{proof}  Let $Det(X)=d$. By definition, for any $\e>0$ there exists a $(d-\e)$-indeterminate measure $\mu$ with a support included in $X$. By Theorem~\ref{indeterminate} we have that $\supp \mu$ contains a $(d-\e)/\pi$-uniform sequence. Therefore, by the gap theorem in~\cite{PolGap} we have that $G(X)\geq 2(d-\e)$. Hence, $G(X)\geq 2Det(X)$.

To prove the other inequality let $G(X)=a$. By the gap theorem, it follows that $X$ contains an $(a-\e)/\pi$-uniform sequence $\G=\{\gamma_n\}$.
 Let $\L=\G\cup \G'$, where $\gamma'_n=(\gamma_n+\gamma_{n+1})/2$. Then $\L$ is a $2(a-\e)/\pi$-uniform sequence and therefore, by theorem~\ref{oscillation},
 $ \MM_{2 a-10\e}(\G,\G')$ contains a non-zero measure. The positive part of this signed measure is $2(a-5\e)/\pi$-determinate and supported on $\G\subset X$. Thus,
  $2Det(X)\geq G(X)$.
\end{proof}

\section{Sign changes of measures with spectral gap}\label{last}

We say that a real measure $\sigma$ has at least one sign change on an interval $(a,b)$ if  there exist Borel sets $P, N\subset [a,b]$ such that $\sigma(P)>0$ and $\sigma(N)<0$. In \cite{EN, EN1}
the number of sign changes of $\sigma$ on an interval $(a,b)$ is  was defined as the minimal degree of a polynomial $p$ for which $pd\sigma$ is a positive measure on $[a,b]$. Clearly one can use either one of these definitions in the statements below.

As a consequence of our oscillation theorem we obtain an improvement of the main result in~\cite{EN, EN1}:

\begin{restatable}{theorem}{Eremenko} \label{Eremenko} If $\sigma$ is a nonzero signed measure with spectral gap $(-a, a)$  then there exists an $a/\pi$-uniform sequence $\{\lambda_n\}$ such that  $\sigma$ has at least one sign change in every $(\lambda_n, \lambda_{n+1})$.
\end{restatable}
\noindent

\begin{proof} Let $A=\supp\mup$, $B=\supp\mum$.  Here, as usual, $\mup$ and $\mum$ denote the positive and negative parts of $\sigma$ in the canonical Jordan decomposition $\sigma=\mup-\mum$. Then $\sigma\in\MM_a(A,B)$. It follows from Lemma~\ref{KrMil} that $\MM_a(A,B)$ contains an extreme point (measure).  By Theorem~\ref{oscillation} this measure is discrete and its support contains a $(a/\pi-\e)$-uniform subsequence for any $\e>0$. Using lemma~\ref{uff}, we can choose an $a/\pi$-uniform subsequence $\L=\{\lambda_n\}$ such that  $\{\lambda_{2n}\}\subset A, \{\lambda_{2n+1}\}\subset B$. Now one can move the points $\lan$, if necessary, so that the sets $A=\supp\mup$ and $B=\supp\mum$ intersect each
interval $[\lambda_n, \lambda_{n+1}]$ and the sequence $\L$ remains $a/\pi$-uniform.
Then $\sigma$ has at least one sign change on each of the intervals $(\lambda_n, \lambda_{n+1})$.
\end{proof}

Next, we  deduce the main result of~\cite{EN, EN1} from corollary~\ref{poissonsummablecor}.

\ErNov*

\begin{proof} Let $\Lambda=\{\lambda_n\}$ be the sequence from corollary \ref{poissonsummablecor}. It is clear that $s(r,\sigma)\geq n_{\Lambda}(r)$, where $n_{\Lambda}(r)$ is the  counting function of $\Lambda$. The fact that $\pi n_\Lambda(x)-\frac a2 x$ is harmonic conjugate of a Poisson-summable function implies that
$$\Pi(\{|\pi n_\Lambda(x)-\frac a2 x|>t\})=o(t)\ \ \textrm{ as }t\to\infty.$$
Since $n_\L$ is a growing function, the last relation implies that $\pi n_\Lambda(x)-\frac a2 x=o(x)$. Therefore $$\lim_{r\to\infty}\frac{n_{\Lambda}(r)}{r}= \frac{a}{2\pi}.$$
\end{proof}
\begin{remark*}
Notice that to prove the last statement we did not use any difficult results of Harmonic Analysis such as the Beurling-Malliavin theory or the
results of \cite{PolGap, PolType}. All in all, our proof seems to be shorter and more elementary than the original proof given in~\cite{EN}.
\end{remark*}

\end{document}